\documentclass[12pt]{amsart}
\usepackage{amssymb,latexsym}
\newtheorem{Theo}{Theorem}
\newtheorem{Lem}[Theo]{Lemma}
\newtheorem{Cor}[Theo]{Corollary}
\newcommand{\N}{\mathbb{N}}
\begin{document}
\baselineskip=17pt
\title{Sumsets avoiding squarefree integers}
\author{Jan-Christoph Schlage-Puchta}
\address{Building S22\\
Krijgslaan 281\\
9000 Gent\\
Belgium}
\email{jcsp@cage.ugent.be}
\date{\today}
\begin{abstract}
We describe the structure of a set of integers $A$ of positive
density $\delta$, such that $A+A$ contains no squarefree integer. It
turns out that the behaviour changes abruptly at the values
$\delta_0=\frac{1}{4}-\frac{2}{\pi^2}=0.0473\ldots$ and
$\delta_1=\frac{1}{18}$.
\end{abstract}
\maketitle
\section{Introduction and results}
Let $A\subseteq[1, x]\cap \N$ be a set of $\delta x$ integers. Erd\H
os\cite{Erd} asked whether $\delta>1/4$ implies that there is some
subset of $A$ adding up to a squarefree integer. Erd\H os and
Freiman\cite{EF} showed that this is indeed the case, in fact,
$\mathcal{O}(\log x)$ summands are already sufficient. This was
further improved by Nathanson and S\'ark\"ozy\cite{NS}, who showed that
21 summands suffice, and Filaseta\cite{Fil} showed that if
$\delta>1-\frac{8}{\pi^2}+\epsilon$, $x>x_0(\epsilon)$ either $A+A$
contains a squarefree integer, or $A$ is a subset of $4\N$ or a
subset of $4\N+2$. Schoen\cite{Sch} showed that if $\delta\geq 0.1$
and $x$ is sufficiently large, then either $A+A$ contains a squarefree
integer, or $A$ is a subset of $4\N$, $9\N$, or $4\N+2$. From this
sequence of results one might expect that for every $\delta>0$ a set
$A$ with $|A|>\delta x$, such that $A+A$ contains no squarefree
integer does so because $A+A$ is not much larger then $A$ itself. However,
here we show that this is only the case for $\delta$ not too
small. More precisely, we show that the values
$\delta=\frac{1}{4}-\frac{2}{\pi^2}=0.0473\ldots$ and
$\delta_1=\frac{1}{18}$ are critical. We will show the following. 
\begin{Theo}
Set $\delta_0=\frac{1}{4}-\frac{2}{\pi^2}$, $\delta_1=\frac{1}{18}$. For every $\epsilon>0$
there exists some $x_0$, such that for all $x>x_0$ the following two
statements become true.
\begin{enumerate}
\item[(i)] Let $A\subseteq[1, x]\cap\N$ be a set of integers, such that
  $A+A$ contains no squarefree integer. If $|A|>(\delta_1+\epsilon)x$,
  then $A$ is a subset of $4\N$, $9\N$, or $4\N+2$.
\item[(ii)] Let $A\subseteq[1, x]\cap\N$ be a set of integers, such that
  $A+A$ contains no squarefree integers. If
  $|A|>(\delta_0+\epsilon)x$, then either $A$ is as in {\rm (i)}, or all
  elements of $A$ are even, and there exist some $a\in\{1, \ldots,
  8\}$, such that for $n\in A$ we have
\begin{eqnarray*}
  n\equiv 2\pmod{4} & \Leftrightarrow & n\equiv a\pmod{9}\\
  n\equiv 0\pmod{4} & \Leftrightarrow & n\equiv -a\pmod{9}.
\end{eqnarray*}
\item[(iii)] There exists a set $A\subseteq\N$ of density
  $>\delta_0-\epsilon$, such that $A+A$ contains no squarefree
  integers, $A+A+A$ contains all integers with finitely many
  exceptions, and the set of not squarefree integers not contained in
  $A+A$ has density at most $\epsilon$.
\end{enumerate}
\end{Theo}

A more qualitative description for infinite sets is given in the
following. For a set 
$A$, let $\Sigma(A)$ be the set of all subset sums of $A$.

\begin{Cor}
Let $A\subseteq\N$ be a set of upper density $\delta$, and assume that
$A+A$ contains no squarefree integers. If $\delta>\frac{1}{18}$, then
$\Sigma(A+A)$ has density at most $\frac{1}{4}$. If
$\delta>\frac{1}{4}-\frac{2}{\pi^2}$, then $\Sigma(A)$ has density at
most $\frac{1}{2}$. On the other hand, there exists a set $A$ of
density $\frac{1}{18}$, such that $A+A$ contains no squarefree
integer, but $A+A+A$ does, and every sufficiently large even integer can be
written as the sum of 6 elements of $A$. Moreover, for every $\epsilon>0$ there
exists a set $A$ of density $\frac{1}{4}-\frac{2}{\pi^2}-\epsilon$,
such that $A+A$ contains no squarefree integer, $A+A$ has density
$1-\frac{6}{\pi^2}-\epsilon$, and $A+A+A=\N$ with finitely many
exceptions. 
\end{Cor}

We remark that Schoen\cite{Sch} constructed a set $A$ of density
$\frac{1}{36}+0.00025\ldots$, such that $A+A$ contains no squarefree
integer, but for any set $B\supseteq A$,
which is periodic, $B+B$ contains squarefree integers.

\section{Proof of part (i) and (ii)}

Let $A\subset[1, x]\cap\N$ be a set, such that
$A>(\delta_0+\epsilon)x$, and that $A+A$ contains no squarefree
integer. Assume further that $A$ is not a subset of $4\N$, $9\N$, or
$4\N+2$. The proof relies on a detailed study of the distribution of
$A$ modulo 36. For each
$a\in\{0, \ldots, 35\}$ set $\delta_a=\frac{36\#\{n\in A, n\equiv
  a\pmod{36}\}}{x}$, that is, $\delta_a$ is the local density of $A$ in
the residue class $a\bmod 36$. Let $U$ be the set of all residue classes $a\bmod
36$ with $\delta_a>1-\frac{9}{\pi^2}+\epsilon/100$, and let $V$ be
the set of all residue classes $a$ with $\delta_a>0$. Finally, let
$Q=\{0, 4, 8, 9, 12, 16, 18, 20, 24, 27, 28, 32\}$ be the set of all
residue classes which do not contain squarefree integers.

\begin{Lem}
We have $U+V\subseteq Q$.
\end{Lem}
\begin{proof}
Suppose that $u\in U, v\in V$, and $a+b\not\in Q$. Fix an integer
$b\in A$ with $b\equiv v\pmod{36}$. Consider the set of all integers
$\{a+b:a\in A, a\equiv u\pmod{36}\}$. None of these integers is
squarefree. However, the density of squarefree integers in a residue
class modulo 36, which is not in $Q$, is
$\frac{4}{3}\cdot\frac{9}{8}\cdot\zeta(2)^{-1}= \frac{9}{\pi^2}$. Since
$\delta_a>1-\frac{9}{\pi^2}+\epsilon/100$, and $x$ is sufficiently
large, this yields a contradiction.
\end{proof}
\begin{Lem}
If $|U|\geq 2$, we are in the situation described in Theorem~{\rm1 (ii)}.
\end{Lem}
\begin{proof}
Suppose first that $U$ contains 3 different elements $a, b, c$. Assume
first that $a$ is odd. Then $2a\equiv 18\pmod{36}$, thus $a$ is 9 or
27. In particular, one of $b$ and $c$ is even, say, $b$ is even. Then
$a+b$ is odd, hence, $b$ is 0 or 18. No matter whether $c$ is odd or
even, $c$ is always divisible by 9. If every element in $V$ is
divisible by 9, then $A\subseteq 9\N$, and we are done. Otherwise
there exists an element $v\in V$ which is not divisible by 9. For this
element we have $a+v, b+v, c+v\in Q$. Since $a$ is odd, this implies that
$v$ is odd, however, then $b+v$ is also odd, hence divisible by 9,
which is impossible since $b$ is divisible by 9, but $v$ is not. Hence, if
$|U|\geq 3$, then all elements in $U$ are even. Moreover, if there is
some odd element $v\in V$, then $U+v$ contains at least 3 odd
elements, which is impossible, since $Q$ contains only 2 odd
elements. Hence $A\subseteq 2\N$. If there are elements $u\in U, v\in
V$ with $u\not\equiv v\pmod{4}$, then we have both $u+v\in Q$ as well as
$u+v\equiv 2\pmod{4}$, which implies $u+v=18$. Hence, for each $v\in V$
there is at most one $u\in U$ with $v\not\equiv u\pmod{4}$. However,
since $|V|\geq|U|\geq 3$, this is 
impossible, and we conclude that all elements in $V$ are congruent
modulo 4. However, this case was excluded in the beginning. Hence, we
find that $|U|\geq 3$ is impossible.

Now consider the case $|U|=2$, and assume first that $|V|\geq 3$. If
one element of $U$ is odd, it has to be 9 or 
27, and $V$ does not contain any even element except possibly 0 and 18. If
$U$ contains another odd element, then $U=\{9, 27\}$. Then $V\setminus
U$ contains at least one odd element $a$, for otherwise we would have
$V\subseteq\{0, 9, 18, 27\}$, which is impossible since
$A\not\subseteq 9\N$. For 
this $a$ we have that both $9+a, 27+a$ are
divisible by 4, which is impossible. Hence, $U$ contains one odd and
one even element, more precisely, we have $U=\{u_1, u_2\}$, where
$u_1\in\{9, 27\}$ and $u_2\in\{0, 18\}$. Let $v\in V\setminus U$ be an
element, which is not divisible by 9. If no such element exists, then
$A\subseteq 9\N$, and we are done. If such an element exists, both
$v+u_1$ and $v+u_2$ are not divisible by 9, and at least one of them
is odd, that is, one of $v+u_1, v+u_2$ is not in $Q$. Hence, all
elements in $U$ are even. 

If an element of $V$ was odd, the
elements in $U$ have difference 18, thus their sum is not divisible by
4, and we conclude that $U=\{0, 18\}$. Hence, all elements in
$V\setminus U$ are odd, which implies that $V\setminus U\subseteq\{9, 27\}$,
contradicting the assumption that $A$ is not contained in
$9\N$. Hence, all elements in $A$ are even. But if the elements are
all even, and not all congruent modulo 4, then there are at least two
sums $u+v$, which are 2 modulo 4, which is impossible. Hence, $|U|=2$
and $|V|\geq 3$ is also impossible.

Hence, it remains to consider the case $U=V=\{u_1, u_2\}$. If $u_1$ is
odd, then $2u_1=18$, that is, $u_1$ is 9 or 27. If $u_2$ was also odd,
then we would have $U=\{9, 27\}$, that is, $A\subseteq 9\N$. If $u_2$
was even, the only possibility is $u_2=18$, which gives the same
contradiction. Hence, both $u_1, u_2$ are even. If they are
congruent to each other modulo 4, we would obtain $A\subseteq 4\N$ or
$A\subseteq 4N+2$. Hence, their sum is an element in $Q$, which is
$2\pmod{4}$, hence, $u_1+u_2\equiv 18\pmod{36}$. But this implies the
description in part (ii).
\end{proof}

\begin{Lem}
The case $U=\emptyset$ is impossible.
\end{Lem}
\begin{proof}
We have
\[
36\delta_0 - 12\big(1-\frac{9}{\pi^2}\big)\geq 24\cdot 0.0269
\]
hence, for every $a\in V$ there exists some $b$, such that
$\delta_b\geq 0.0269$, and $a+b\not\in Q$. In particular, we can find
elements $b_1, b_2\in A$ with $|b_1-b_2|\leq 40$ and $b_i\equiv
b\pmod{36}$. As $x\rightarrow\infty$, the number of integers $n\equiv
a\pmod{36}$, such that both $n+b_1$ and $n+b_2$ are not squarefree, is
asymptotically equal to
\begin{multline*}
\frac{x}{36}\left(1-2\prod_{p\geq 5}\big(1-\frac{1}{p^2}\big) +
\underset{p^2|b_1-b_2}{\prod_{p\geq 5}}\big(1-\frac{1}{p^2}\big)
\underset{p^2\nmid b_1-b_2}{\prod_{p\geq 5}}\big(1-\frac{2}{p^2}\big)\right)\\
\leq \frac{x}{36}\left(1-\frac{18}{\pi^2}+\frac{24}{25}\prod_{p\geq 7}
\big(1-\frac{2}{p^2}\big)\right) = 0.04266\ldots\cdot\frac{x}{36},
\end{multline*}
hence, for $x$ sufficiently large we obtain $\delta_a\leq 0.04267$ for
all $a$. But this is 
impossible, since $\delta_0=0.0473>0.04267$.
\end{proof}
\begin{Lem}
\label{Lem:38}
Let $a, b\in V$ be classes with $a+b\not\in Q$, and let
$\epsilon>0$ be given. Then for 
$N$ sufficiently large we have $\delta_a+\frac{3}{8}\delta_b\leq
1-\frac{9}{\pi^2}+\epsilon$. 
\end{Lem}
\begin{proof}
Since $a, b\in V$ means that $\delta_a, \delta_b$ are positive, we see
that none of $a, b$ is in $U$.

For the proof we will show that a lower bound for $\delta_b$
implies that $A$ contains elements $b_1, b_2\equiv b\pmod{36}$ with
difference not divisible by too many prime squares. This will then
yield an upper bound for $\delta_a$, and comparing the bounds yields
our claim.

Let $k\geq 3$ be the least integer, such that $A$ contains elements
$b_1, b_2$ with $p_k^2\nmid b_1-b_2$. Then $\delta_b\leq
\prod_{i=3}^{k-1} p_i^{-2}+\epsilon$, and $\delta_a$ is at  most
$\epsilon$ plus the density of integers $n\equiv a\pmod{36}$, such
that both $n+b_1$ and $n+b_2$ are not squarefree. The density of the
integers $n$ such that both $n+b_1$ and $n+b_2$ are squarefree equals
\[
\underset{p^2|b_1-b_2}{\prod_{p\geq 5}}\big(1-\frac{1}{p^2}\big)
\underset{p^2\nmid b_1-b_2}{\prod_{p\geq 5}}\big(1-\frac{2}{p^2}\big)
\leq \big(1-\frac{1}{p_k^2-1}\big)\frac{9}{\pi^2}.
\]
The density of integers $n$ such that none of $n+b_1$, $n+b_2$ is
squarefree is therefore at most
\[
1-\frac{18}{\pi^2} + \big(1-\frac{1}{p_k^2-1}\big)\frac{9}{\pi^2} =
1-\frac{9}{\pi^2} -\frac{9}{\pi^2(p_k^2-1)}.
\]
Hence, 
\[
\delta_a+\frac{3}{8}\delta_b \leq 1-\frac{9}{\pi^2} -\frac{9}{\pi^2(p_k^2-1)} +
\frac{3}{8}\prod_{i=3}^{k-1} p_i^{-2}+\epsilon \leq 1-\frac{9}{\pi^2}+\epsilon,
\]
provided that
\[
\frac{9}{\pi^2(p_k^2-1)} \geq \frac{3}{8}\prod_{i=3}^{k-1} p_i^{-2},
\]
which follows for $k\geq 5$ from $p_k<2p_{k-1}$, and for $k=4$ by
direct inspection. Hence, we find that $k=3$, and obtain $\delta_a\leq
1-\frac{9}{\pi^2} - 0.379$. Thus the bound $\delta_a+\frac{3}{8}\delta_b\leq
1-\frac{9}{\pi^2}+\epsilon$ could only fail if $\delta_b>0.101$. But
then there are elements $b_1, b_2\in A$, $b_1, b_2\equiv b\pmod{36}$,
with $|b_1-b_2|<25$, thus, $b_1-b_2$ is not divisible by the square of
any prime different from 2, 3. This implies
\[
\delta_a \leq 1-\frac{18}{\pi^2} + \prod_{p\geq
  5}\big(1-\frac{2}{p^2}\big)+\epsilon \leq 0.0066,
\]
and the relation $\delta_a+\frac{3}{8}\delta_b\leq
1-\frac{9}{\pi^2}+\epsilon$ holds unless $\delta_b>0.217$, which is
impossible since we already know that $b\not\in U$.
\end{proof}
\begin{Lem}
Suppose that $2a\not\in Q$. Then $\delta_a\leq 0.04+\epsilon$.
\end{Lem}
\begin{proof}
If $\delta_a>0.04+\epsilon$, then there exist elements $a_1, a_2\in
A$, $a_1, a_2\equiv a\pmod{36}$ with $|a_1-a_2|<25$. Hence, $a_1-a_2$
is not divisible by the square of any prime different from 2 and 3, and
the same argument as in the previous lemma now implies $\delta_a\leq
0.0066$, which proves our claim.
\end{proof}

\begin{Lem}
The case $|U|=1$ is impossible.
\end{Lem}
\begin{proof}
If $U=\{u\}$, then $A$ is concentrated in 12 residue classes modulo 36,
more precisely, we have $V\subseteq Q-u$. On the other hand, we have
\[
\delta_0 = \frac{1}{36}\big(1+8(1-\frac{9}{\pi^2})\big),
\]
and therefore $V\geq 9$. Assume first that $U$ is odd. Then at most 2
elements in $V$ are even, hence, $V$ contains at most four elements
$v$ with $2v\in Q$. Hence, we have $\delta_u\leq 1$, $\delta_a\leq
1-\frac{9}{\pi^2}+\epsilon$ for three more classes $a_1, a_2, a_3\in V$, and
$\delta_a\leq 0.04+\epsilon$ for the remaining 8 classes. Hence,
\[
|A| \leq \frac{x}{36}\big(1+3(1-\frac{9}{\pi^2}) +
8\cdot 0.04+\epsilon\big) \leq 0.0441 x,
\]
which gives a contradiction since $\delta_0=0.0473\ldots$. Every
$v\in V$ satisfies precisely one of the two relations $v+u \equiv
0\pmod{4}$ or  $v+u \equiv 0\pmod{9}$. Let $v_1, v_2\in
V\setminus\{u\}$ be residue classes, such that $v_1$ satisfies the first,
and $v_2$ the second condition. Then $v_1+v_2\equiv 4x+9y+2u\pmod{36}$, where
$x=\frac{v_1-u}{4}\not\equiv 0\pmod{9}$, and
$y=\frac{v_2-u}{9}\not\equiv 0\pmod{4}$. Since $u$ is 
even, the right hand side is not divisible by 4. If it was divisible
by 9, then we obtain the relations $v_1+v_2\equiv u+v_2\equiv
0\pmod{9}$, that is, $v_1\equiv u\pmod{9}$. Together with $v_1+u\equiv
0\pmod{4}$ and the fact that $u$ is even we obtain $v_1\equiv
u\pmod{36}$, a contradiction. Hence, if we set $V_1=\{v\in V:v\neq u,
4|v-u\}$, and $V_2=\{v\in V:v\neq u, 9|v-u\}$, then $V_1+V_2\cap
Q=\emptyset$. Set $\delta_1=\max_{v\in V_1}\delta_v$,
$\delta_2=\max_{v\in V_2}\delta_v$. Then from Lemma~\ref{Lem:38} we
find that $\delta_1+\frac{3}{8}\delta_2\leq
1-\frac{9}{\pi^2}+\epsilon$, hence,
\[
|A|\leq \frac{x}{36}\big(1+8\delta_1+3\delta_2) =
\frac{x}{36}\big(1+8(\delta_1+3\frac{3}{8}\delta_2)\big) \leq
\frac{x}{36}\big(1+8(1-\frac{9}{\pi^2})+\epsilon\big) = 
(\delta_0+\epsilon)x.
\]
Hence, our claim follows.
\end{proof}

\section{Proof of part (iii) and the corollary}

Our proof is constructive. Denote by $p_i$ the $i$-th prime
number. Fix an integer $k\geq 2$, and set $q=\prod_{i=2}^k
p_i^2$. Then we define 
\[
\mathcal{A}(k)=\{n: 4|n, \mu^2((n, q))=0\}\cup \{n: q|n\} =
\mathcal{A}_1\cup\mathcal{A}_2,
\]
say. We claim that  $\mathcal{A}+\mathcal{A}$ does not contain
squarefree integers. In fact, every element of $A_1+A_1$ is divisible
by 4, while every element of $A_2+A_2$ is divisible by $q$, and $q$ is
obviously a square. If $x\in A_1, y\in A_2$, then $(x+y, q) = (x, q)$
is not squarefree, that is, $x+y$ has a divisor, which is not
squarefree, and is therefore again not squarefree.

The density of $A_1$ is
\[
\frac{1}{4}\Bigg(1-\prod_{i=2}^k\big(1-\frac{1}{p_i^2}\big)\Bigg),
\]
which converges for $k\rightarrow\infty$ to
\[
\frac{1}{4}-\frac{1}{3}\zeta(2)^{-1} = \frac{1}{4}-\frac{2}{\pi^2}
= \delta_0.
\]

Next, we show that $A+A+A$ contains all but finitely many
integers. To do so note that since $k\geq 2$ we now that $A_1$
contains all multiples of 36 as well as all multiples of 100. Hence,
$A_1+A_1$ contains all integers divisible by 4 with finitely many
exceptions, the largest of which is $764$. $A_2$ contains all integers
divisible by $q$, and since $q$ is odd, we see that $0, q, 2q, 3q$
are different residues modulo 4, thus, every integer $> 764 + 3q$ can
be written as the sum of a multiple of 36, a multiple of 100, and a
multiple of $q$, thus, $A_1+A_1+A_2$ contains all integers with
finitely many exceptions.

Finally, we have to bound the number of not squarefree integers, which
are not contained in $A+A$. Let $n$ be an integer, divisible by
$p^2$. If $p=2$, and $n\geq 764$, we can write $n=36x+100 y$, thus
$n\in A+A$. If $p=p_i$, $2\leq i\leq k$, and $n>4q$, we can write $n$
as $4p^2 x + q y$, thus $n\in A+A$. Hence, the density of not
squarefree numbers not contained in $A+A$ is at most
$\sum_{i>k}\frac{1}{p_i^2}$, which tends to 0 as $k$ goes to infinity.

We now prove the corollary.
The only thing which is not obvious from the theorem are the claimed
properties of the set constructed in Theorem~1 (ii). We take
$a=1$. Then every integer $n\equiv 30\pmod{36}$ can be written as the
sum of 3 elements of $A$, and this residue class contains infinitely
many squarefree integers. To show that every sufficiently large even
integer is the sum of 6 elements in $A$ it suffices to show that every
sufficient large integer, which is divisible by 4, is the sum of 5
elements of $A$, that is, we have to show that every residue class $n$ 
modulo 9 can be written as $n=2x-y$, where $x, y$ are natural numbers
with $0<2x+y\leq 5$. However, this can be checked with not much effort.

\end{document}